\documentclass[11pt]{article}
\usepackage[margin=1in]{geometry}
\usepackage{amsmath,amssymb,amsthm}
\usepackage{graphicx}
\usepackage{enumitem}
\usepackage{microtype}
\usepackage[hidelinks]{hyperref}

\newtheorem{theorem}{Theorem}
\newtheorem{lemma}{Lemma}
\newtheorem{question}{Question}

\title{Euclidean vs Graph Metric: The Fixed-Source Problem}
\author{Itai Benjamini}
\date{}

\begin{document}
\maketitle

\begin{abstract}
We prove that two fixed sources in the Euclidean plane can be realized by a bounded-degree planar
unit-edge graph on a $10$-net, with graph distance from each source agreeing with Euclidean distance
up to a universal additive constant. We ask whether the analogous statement holds for three
non-collinear sources, and prove a logarithmic obstruction for large ordered source sets in the
coordinate-planar setting.
\end{abstract}

\section{Introduction}

All graph edges below have graph length one. A subset $N\subset \mathbb R^2$ is called a $10$-net if every
point of $\mathbb R^2$ lies within Euclidean distance $10$ of some point of $N$.

Let
\[
S=\{p_1,\ldots,p_m\}\subset \mathbb R^2.
\]
A fixed-source version of the slack-Euclidean problem~\cite{PPS90,Radin94,Benjamini13} asks whether one can find
 a bounded-degree planar graph $G$, drawn in the plane, whose vertex set $N$ is a $10$-net containing $S$, and such that
\[
 d_G(p_i,x)=|p_i-x|+O(1)
\]
for every $x\in N$ and every source $p_i\in S$. The additive constant is required to be independent of scale.

One source is easy: a radial bounded-degree tree gives the required distances. The purpose of this note is to
record that two sources are also possible. We also record a complementary obstruction: in
coordinate-planar settings there is no uniform bounded-additive theorem for arbitrary finite source
sets. The obstruction already appears for large ordered source families.

\begin{theorem}\label{thm:main}
For every two points $p,q\in\mathbb R^2$ there exists a bounded-degree planar graph $G$, drawn in
$\mathbb R^2$, whose vertex set is a $10$-net containing $p$ and $q$, such that
\[
 d_G(p,x)=|p-x|+O(1),\qquad d_G(q,x)=|q-x|+O(1)
\]
for every vertex $x$ of $G$. The constants are universal.
\end{theorem}

The proof is based on confocal coordinates for the two foci $p$ and $q$. The graph is a thickened
confocal grid: large confocal cells carry large clouds of vertices, and adjacent clouds are connected by
noncrossing monotone ladders.

\section{Confocal coordinates}

After a translation and rotation, write
\[
p=(0,0),\qquad q=(D,0),
\]
where first we assume that $D$ is a large integer. The cases of bounded $D$ and non-integral $D$ are handled at the end.

For $x\in\mathbb R^2$, define
\[
t(x)=\frac{D+|x-p|-|x-q|}{2},\qquad
h(x)=\frac{|x-p|+|x-q|-D}{2}.
\]
Then
\[
0\le t(x)\le D,\qquad h(x)\ge 0,
\]
and, crucially,
\begin{equation}\label{eq:linearization}
|x-p|=t(x)+h(x),\qquad |x-q|=D-t(x)+h(x).
\end{equation}
Thus the two distance functions are linear in the coordinates $(t,h)$.

For $0\le t\le D$ and $h\ge 0$, the corresponding point in the upper or lower half-plane is
\begin{equation}\label{eq:param}
X=t+\frac{(2t-D)h}{D},\qquad
Y=\pm \frac{2}{D}\sqrt{t(D-t)h(D+h)}.
\end{equation}
Hence the curves $h=\mathrm{constant}$ are confocal ellipses and the curves $t=\mathrm{constant}$ are confocal hyperbolas.

\begin{figure}[ht]
\centering
\includegraphics[width=.84\textwidth]{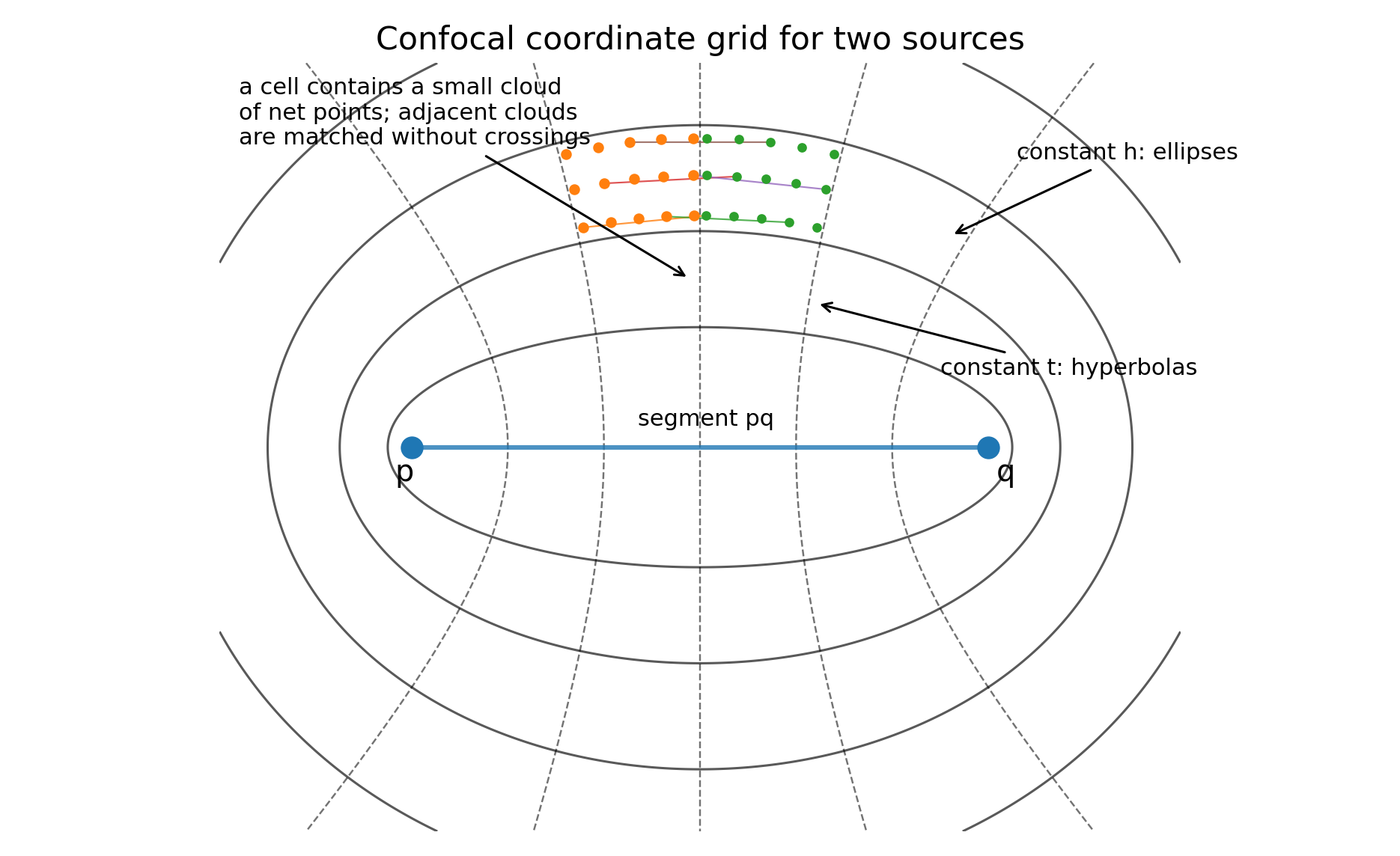}
\caption{A confocal coordinate grid for the two sources $p$ and $q$. The graph is built by thickening this grid.}
\label{fig:confocal}
\end{figure}

We now divide the $(t,h)$ half-strip into unit rectangles
\[
R_{i,j}=[i,i+1]\times[j,j+1],\qquad 0\le i<D,\qquad j=0,1,2,\ldots.
\]
The image of such a rectangle under~\eqref{eq:param} is called a \emph{confocal cell}. Except on the line $h=0$, each rectangle has an upper and a lower copy.

\section{How the short graph paths look}

Before giving the full construction, let us describe the shape of the short path from $p$ to a vertex
$x$ in a confocal cell $(i,j)$.  The distance proof below shows that such a path is a shortest path up to a universal additive error.  The path has two components of motion.

\begin{enumerate}[label=(\arabic*)]
\item A \emph{horizontal motion} in the bottom row $j=0$, moving from $p$ to the correct column $i$.
\item A \emph{vertical motion} in the fixed column $i$, climbing from level $0$ to level $j$.
\end{enumerate}

Thus this short path reaches the correct column first, and then climbs through successive clouds in that
column. Since the first phase uses about $i$ steps and the second phase uses about $j$ steps, its total length is
about $i+j$, in agreement with~\eqref{eq:linearization}.

The reason the clouds become larger as one climbs in a fixed column is simple: the Euclidean image of a confocal
cell may become wider, so more vertices are needed in that cell in order for the full vertex set to remain a $10$-net.
The next figure shows a typical column, the noncrossing ladders between successive clouds, and one such short path from
$p$ to a vertex $z$ in the top cloud.

\begin{figure}[ht]
\centering
\includegraphics[width=.92\textwidth]{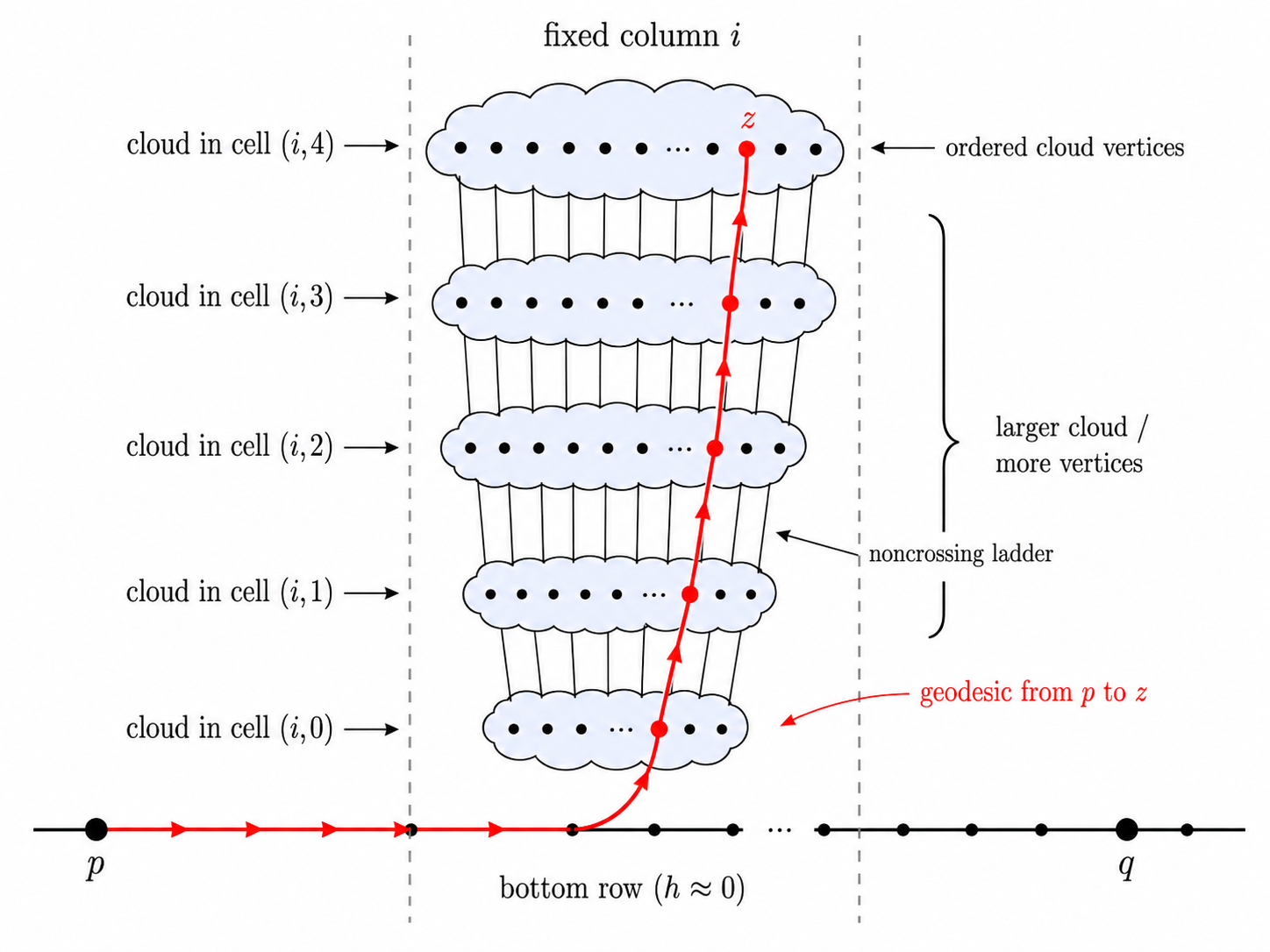}
\caption{A fixed column. The clouds may increase in size as the Euclidean cells become larger. A shortest path up to $O(1)$ from $p$ first moves along the bottom row and then climbs through the noncrossing ladders to reach a chosen vertex $z$.}
\label{fig:column}
\end{figure}

\section{The combinatorial ladders}

The planar connections between neighboring clouds are built from a simple ordered-ladder lemma.

\begin{lemma}[Ordered ladder]\label{lem:ladder}
Let $P$ and $Q$ be two finite linearly ordered sets satisfying
\[
C_0^{-1}|P|\le |Q|\le C_0|P|.
\]
Place $P$ and $Q$ on opposite sides of a rectangle, in their given orders. Then there is a noncrossing
bipartite graph between $P$ and $Q$ such that every vertex is incident to at least one edge and the maximum degree is bounded by a constant depending only on $C_0$.
\end{lemma}

\begin{proof}
Write $P=\{p_1,\ldots,p_m\}$ and $Q=\{q_1,\ldots,q_n\}$ in order. Associate to $p_r$ the interval
$[(r-1)/m,r/m]$ and to $q_s$ the interval $[(s-1)/n,s/n]$ in $[0,1]$. Connect $p_r$ to $q_s$ whenever the two intervals intersect.
Then the edges are order-preserving and hence can be drawn without crossings. Every interval intersects at least one interval from the other partition, so every vertex has a neighbor. Since $m/n$ is bounded above and below in terms of $C_0$, the degrees are bounded in terms of $C_0$.
\end{proof}

\section{Clouds in the confocal cells}

We next place vertices in the confocal cells.  The geometric point to check is that neighboring cells require comparable numbers of net points.  This follows from the explicit parametrization.  Put
\[
        u=\frac{2t}{D}-1,
        \qquad
        v=1+\frac{2h}{D}.
\]
Then the confocal map is the standard elliptic coordinate map
\[
        X-\frac D2=\frac D2uv,
        \qquad
        Y=\pm\frac D2\sqrt{(1-u^2)(v^2-1)}.
\]
In these coordinates the Euclidean metric is
\begin{equation}\label{eq:elliptic-metric}
        ds^2=
        \frac{v^2-u^2}{1-u^2}\,dt^2+
        \frac{v^2-u^2}{v^2-1}\,dh^2 .
\end{equation}
The coefficients have only the endpoint singularities
\[
        t^{-1/2},\qquad (D-t)^{-1/2},\qquad h^{-1/2}
\]
when side lengths are integrated over a unit interval.  These singularities are integrable, and their integrals over adjacent unit intervals are comparable by a universal constant.  Consequently, the Euclidean side lengths and areas of the images of adjacent unit rectangles are comparable up to a universal factor.  In particular, the number of Euclidean balls of radius $2$ needed to cover one confocal cell differs from the corresponding number for any adjacent confocal cell by at most a universal factor.  The finitely many cells in the bounded exceptional region where $D$ is small or where the first and last partial columns are merged are handled by changing the constants.

\begin{lemma}[Cloud realization]\label{lem:clouds}
There is a choice of finite clouds $P_C$ inside the confocal cells $C$ such that:
\begin{enumerate}[label=(\roman*)]
\item the union of all clouds, together with $p$ and $q$, is a $10$-net of the plane;
\item if two confocal cells are adjacent, then their cloud sizes are comparable by a universal factor;
\item the cloud in each cell is ordered in each direction in which it is connected to a neighboring cloud;
\item the ladders between adjacent clouds can be drawn in pairwise disjoint thin corridors near the relevant cell sides, so that all ladders together form a planar bounded-degree graph;
\item for every ladder from a preceding cloud to a following cloud, each vertex of the following cloud has at least one predecessor in the preceding cloud.
\end{enumerate}
\end{lemma}

\begin{proof}
For each confocal cell $C$, let $M(C)$ be the smallest number of Euclidean balls of radius $2$ needed
to cover $C$. By the estimate preceding the lemma, $M(C)$ and $M(C')$ are comparable by a universal
factor whenever $C$ and $C'$ are adjacent. Choose a finite set $P_C\subset C$ of cardinality comparable
to $M(C)$ which is $5$-dense in $C$; for example, take a maximal $2$-separated set and then change its
size, if necessary, by only a bounded factor. This gives a $10$-net and gives adjacent clouds with
comparable cardinalities.

Each non-exceptional cell is a topological quadrilateral. For a vertical ladder, order the cloud by its
position along the common horizontal side; for a horizontal bottom-row ladder, order the cloud by its
position along the common vertical side. A bottom-row cloud may carry both a horizontal order and a
vertical order. This causes no conflict: reserve inside each bottom-row cell two disjoint narrow
corridors, one for the horizontal ladder along the bottom row and one for the vertical ladder leaving the
cell. Away from the bottom row only vertical column ladders are used. The exceptional cells near $p$
and $q$ are replaced by a fixed finite planar gadget, so they do not affect the constants.

Now apply Lemma~\ref{lem:ladder} to every adjacent pair of clouds that must be connected. Since
adjacent cloud sizes are comparable, every ladder has uniformly bounded degree. Since every ladder is
order-preserving and is drawn inside its own corridor, different ladders are disjoint except at their
intended endpoint vertices. Hence the total drawing is planar and has bounded degree. Finally, the
construction in Lemma~\ref{lem:ladder} makes every vertex of the following cloud incident to at least
one edge from the preceding cloud. This is the predecessor property used below to construct the short
paths from $p$ and $q$.
\end{proof}

\section{Proof of the theorem}

We now construct the graph $G$.

\smallskip
\noindent\textbf{Bottom row.}
For $j=0$, connect the clouds in consecutive cells $(i,0)$ and $(i+1,0)$ by the ordered ladder lemma, separately in the upper and lower half-planes. Connect the first clouds near $p$ to $p$ by boundedly many edges, and the last clouds near $q$ to $q$ by boundedly many edges.

\smallskip
\noindent\textbf{Vertical columns.}
For each fixed $i$ and every $j\ge 1$, connect the cloud in cell $(i,j-1)$ to the cloud in cell $(i,j)$ by the ordered ladder lemma, again separately in the upper and lower half-planes.

All edges have graph length one.

A vertex in a cell with indices $(i,j)$ is assigned two labels
\[
A=i+j,\qquad B=D-i+j.
\]
By~\eqref{eq:linearization}, for every vertex $x$ in this cell we have
\begin{equation}\label{eq:ABmetric}
|x-p|=A+O(1),\qquad |x-q|=B+O(1).
\end{equation}

Every graph edge either moves horizontally between consecutive first-row cells or vertically between consecutive cells in a fixed column. Therefore each edge changes both labels $A$ and $B$ by at most one. It follows that every path from $p$ to $x$ has length at least $A-O(1)$, and every path from $q$ to $x$ has length at least $B-O(1)$.

For the reverse inequalities, we use the reachability built into the ladders. From $p$, first move along the bottom row to column $i$. Since every vertex in each first-row cloud has a predecessor in the preceding cloud, every vertex in the first-row cloud of column $i$ is reachable in $i+O(1)$ steps. Then climb the column using the vertical ladders. Again, every vertex in the cloud at level $j$ has a predecessor in the cloud at level $j-1$, so induction on $j$ yields a path from $p$ to every vertex in cell $(i,j)$ of length
\[
i+j+O(1)=A+O(1).
\]
Similarly, starting from $q$ and moving along the bottom row from the other side yields a path of length
\[
D-i+j+O(1)=B+O(1)
\]
from $q$ to every vertex in cell $(i,j)$.

Combining these bounds with~\eqref{eq:ABmetric}, we obtain
\[
d_G(p,x)=A+O(1)=|p-x|+O(1),
\]
and
\[
d_G(q,x)=B+O(1)=|x-q|+O(1).
\]
This proves Theorem~\ref{thm:main} for large integral $D$.

If $D$ is non-integral, replace $D$ by $\lceil D\rceil$ in the labels and merge the final partial column with the previous one; this changes only the additive constants. If $D$ is bounded, use a one-source radial construction from $p$ and connect $q$ to $p$ by a path of bounded length. Since $|p-q|=O(1)$, the distance estimates from $q$ follow from those from $p$, up to another universal additive constant. Thus the theorem holds in full generality.

\section{The next fixed-source question}

The two-source construction uses the special fact that the two Euclidean distance functions from $p$ and $q$
become linear in the two coordinates $(t,h)$. There is no analogous two-coordinate linearization for three non-collinear sources.

\begin{question}[Three fixed sources]
Let $p_1,p_2,p_3$ be three non-collinear points in $\mathbb R^2$. Does there exist a bounded-degree planar graph $G$, drawn in the plane, whose vertex set is a $10$-net containing $p_1,p_2,p_3$, such that
\[
d_G(p_i,x)=|p_i-x|+O(1)
\]
for every vertex $x$ and for $i=1,2,3$?
\end{question}

More generally, for which finite source sets $S\subset\mathbb R^2$ can one build a bounded-degree planar graph on a net that approximates the distance functions from all sources in $S$ up to bounded additive error?

\section{Many ordered sources: a logarithmic obstruction}

The two-source theorem does not extend uniformly to arbitrary finite source sets in settings where the
Euclidean drawing gives the planar order used below.  The obstruction is the same Monge-discrepancy mechanism that appears in the
usual planar boundary-distance argument.

We formulate the statement in a way that isolates the topological input.  Say that two ordered source
families
\[
        a_0,\ldots,a_{L-1},\qquad b_0,\ldots,b_{L-1}
\]
have the \emph{planar order property} if for every $i,j$ a shortest path from $a_i$ to $b_{j+1}$ and a
shortest path from $a_{i+1}$ to $b_j$ meet.  Then swapping tails gives
\begin{equation}\label{eq:monge}
        D(i,j)+D(i+1,j+1)\le D(i,j+1)+D(i+1,j),
\end{equation}
where $D(i,j)=d_G(a_i,b_j)$.  This property holds, for example, in the usual coordinate-planar situation
when the four endpoints occur in crossed order on the boundary of a topological rectangle and the relevant
geodesics stay inside that rectangle.

\begin{theorem}[A logarithmic obstruction for many ordered sources]\label{thm:many-sources}
There is a constant $c>0$ with the following property.  For arbitrarily large $k$ there is a set of
$k$ points $S\subset\mathbb R^2$ such that the following holds.  Suppose $G$ is a unit-edge graph drawn
planarly in the same Euclidean plane, $S\subset V(G)$, the two ordered subfamilies of $S$ defined below have the planar
order property, and
\[
        |d_G(s,x)-|s-x||\le C
\]
for every $s\in S$ and every vertex $x$ of $G$.  Then
\[
        C\ge c\log k .
\]
\end{theorem}

\begin{proof}
Let $N$ be large and put $L=\lfloor N^{2/3}\rfloor$.  Take the source set
\[
        a_i=(0,i),\qquad i=0,\ldots,L-1,
\]
and
\[
        b_j=(N,N+j),\qquad j=0,\ldots,L-1.
\]
Thus $k=2L$.  Since all $a_i$ and $b_j$ are sources, the assumed distance estimate gives
\[
        D(i,j):=d_G(a_i,b_j)
        =E(i,j)+O(C),
\]
where
\[
        E(i,j)=|a_i-b_j|=\sqrt{N^2+(N+j-i)^2}.
\]
By the planar order property, $D$ satisfies the Monge inequality~\eqref{eq:monge}.  Hence
\[
        \mu(i,j)=D(i,j+1)+D(i+1,j)-D(i,j)-D(i+1,j+1)
\]
is a nonnegative integer measure on the index square.

For the Euclidean matrix $E$, the corresponding mixed difference
\[
        \mu_E(i,j)=E(i,j+1)+E(i+1,j)-E(i,j)-E(i+1,j+1)
\]
has size comparable to $1/N$ throughout this $L\times L$ square.  More precisely, Taylor expansion gives
\[
        \mu_E(i,j)=\alpha_N+O(L/N^2),
        \qquad \alpha_N\asymp 1/N .
\]
Therefore, for every axis-parallel rectangle $R$ in the index square,
\[
        \sum_R \mu_E=\alpha_N |R|+O(1).
\]
On the other hand, rectangle sums of $\mu$ telescope to four corner values of $D$, and the same is true
for $\mu_E$.  Since $D=E+O(C)$, we get
\[
        \left|\sum_R \mu-\alpha_N |R|\right|\le O(C+1)
\]
for every such rectangle $R$.

The total mass is
\[
        \alpha_N L^2\asymp N^{1/3}.
\]
Thus $\mu$ is a nonnegative integer measure of mass tending to infinity whose discrepancy for
axis-parallel rectangles is $O(C+1)$.  Schmidt's logarithmic lower bound for rectangle discrepancy, in
one of its standard forms as recorded for instance in Matou\v{s}ek's book~\cite{Matousek99}, implies
\[
        C+1\ge c'\log(N^{1/3}).
\]
Since $k=2L\asymp N^{2/3}$, this gives $C\ge c\log k$ after changing the absolute constant.
\end{proof}

The theorem shows that there is no uniform positive answer for all finite source sets in any class of
planar embeddings where large ordered source families force the Monge inequality.  It does not address
the fixed case of three non-collinear sources.

\section*{Acknowledgments}
The author thanks ChatGPT for useful help.

\end{document}